\documentclass[12pt]{article}
\usepackage{epsfig}
\usepackage{fancyhdr}
\usepackage{amssymb}
\usepackage{amsmath}

\usepackage{graphicx}
\usepackage[usenames,dvipsnames]{color}

\newtheorem{theorem}{Theorem}
\newtheorem{lemma}[theorem]{Lemma}

\newenvironment{proof}{\noindent{\bf Proof}\hspace{0.5em}}
    { \null  \hfill $\square$ \par}

\usepackage[mathscr]{eucal}

\renewcommand{\S}{\mathcal S}

\newcommand{\ES}{{\mathbb X}}

\newcommand{\EY}{\mathbb Y}
\newcommand{\EZ}{\mathbb Z}

\newcommand{\st}{:}

\newcommand\GF{{\rm GF}}
\newcommand\PG{{\rm PG}}

\renewcommand\setminus{\backslash}

\usepackage{setspace}

\begin{document}
%
%

\title{Hyper-reguli in $\PG(5,q)$}

 \date{\today}

\author{S.G. Barwick and Wen-Ai Jackson
\date{\today}
\\ Department of Pure Mathematics, University of Adelaide\\
Adelaide 5005, Australia
}

\maketitle

Corresponding Author: Dr Susan Barwick, University of Adelaide, Adelaide
5005, Australia. Phone: +61 8 8313 3983, Fax: +61 8 8313 3696, email:
susan.barwick@adelaide.edu.au



\begin{abstract} 
A simple counting argument is used to show that for all $q$, an Andr\'e hyper-regulus $\ES$ in $\PG(5,q)$ has exactly two switching sets. 
Moreover, there are exactly $2(q^2+q+1)$ planes in $\PG(5,q)$ that meet every plane of $\ES$ in a point, namely the planes in the switching sets. 
\end{abstract}

Bruck~\cite{bruc73b} investigated covers of the circle geometry $CG(3,q)$ and showed a cover corresponds to a set $\ES$ in  $\PG(5,q)$. The set $\ES$  consists  of $q^2+q+1$
 mutually disjoint planes, and  has 
  two {\em switching sets} $\EY,\EZ$, with the property that two planes from different sets ($\ES$, $\EY$ or $\EZ$) meet in a unique point, and two planes from the same set are disjoint. 
Ostrom \cite{ostrom} called these three sets {\em Andr\'e hyper-reguli}, and posed the question of whether there are any other hyper-reguli that cover the same set of points as $\ES$. Pomareda \cite{poma97} used algebraic techniques to show that there are no other switching sets of $\ES$ when 
 the largest power of $3$ that divides $q-1$ is 1. In this article we
prove that this result holds for all $q$ using a simple counting argument. Moreover, we 
show  that not only are there no more hyper-reguli for $\ES$, but there are no further planes in $\PG(5,q)$ that meet each plane of $\ES$ in a point.

\begin{lemma}\label{numberofcovers}
The number of covers of the circle geometry $CG(3,q)$ is $\frac12q^3(q-1)(q^3+1).$
\end{lemma}
\begin{proof}
In \cite{bruc73b}, it is shown that every cover of $CG(3,q)$ can be represented as either 
I: $\{x\in\GF(q^3)\st N(x-a)=f\}$, for some $a\in\GF(q^3)$, and $f\in\GF(q)\setminus\{0\}$; or II: 
$\{x\in\GF(q^3)\cup\{\infty\}\st N\left(\frac{x-a}{x-b}\right)=f\}$, for some $a\neq b\in\GF(q^3)$, and $f\in\GF(q)\setminus\{0\}$,
where $N$ is the norm from $\GF(q^3)$ to $\GF(q)$, that is $N(x)=x^{q^2+q+1}$. We count the number of covers, firstly,
there are $q^3(q-1)$ covers of type I. 
To count covers of type II, note that if $N((x-a)/(x-b))=f$, then $N((x-b)/(x-a))=1/f$. So the number of covers of type II is $\frac12q^3(q^3-1)(q-1)$. 
\end{proof}

\begin{theorem}\label{covers-unique}
 Let $\ES$ be an Andr\'e hyper-regulus in $\PG(5,q)$. Then there are {\em exactly} $2(q^2+q+1)$ planes that meet every plane of $\ES$, namely the planes in the two switching sets of $\ES$. 
\end{theorem}

\begin{proof} 
Let $\S$ be a regular 2-spread containing $\ES$. By \cite{bruc73b}, the planes of $\S$ correspond to points of the circle geometry $CG(3,q)$, and the covers of $CG(3,q)$ are equivalent to the  Andr\'e hyper-reguli contained in $\S$.
Moreover by \cite{bruc73b},  given an Andr\'e hyper-reguli 
$\ES$ contained in $\S$, there are at least $2(q^2+q+1)$ planes
that meet every plane of $\ES$ in a point (namely the planes in the two switching sets). The number of covers of $CG(3,q)$ is counted in Lemma~\ref{numberofcovers}, hence the number $x$ of planes of $\PG(5,q)$ that meet $q^2+q+1$ planes of $\S$ in a point is
at least $y=\frac12 q^3(q-1)(q^3+1)\times 2(q^2+q+1)$.

We now count the number $x$ exactly. As the 2-spread $\S$ covers all the points of $\PG(5,q)$, 
we can partition the planes of $\PG(5,q)$  into three types with respect to $\S$. Type A consists of planes of $\S$; Type B consists of planes that meet $q^2+q+1$ elements of $\S$ in exactly one point; and Type C consists of planes that  meet one element of $\S$ in a line (and so meet $q^2$ elements of $\S$ in exactly one point).  
We note that planes of type B and C determine linear sets of $\PG(1,q^3)$ of rank 3.
We count the number of planes of each type. There are $q^3+1$ planes of type A. To count planes of type C, note that  there are $q^3+1$ choices for a plane of $\S$, each contains $q^2+q+1$ lines, and each of these lies in $q^3+q^2+q$ planes not in $\S$. Hence there are $q(q^3+1)(q^2+q+1)^2$ planes of type C. 
The total number of planes in $\PG(5,q)$ is $(q^3+1)(q^2+1)(q^4+q^3+q^2+q+1)$. Hence the remaining 
$x=q^3(q^3+1)(q^3-1)$ planes are of type B. 

Thus we have $x=y$, that is, each plane of $\PG(5,q)$ that meets  $q^2+q+1$ planes of $\S$ lies in an Andr\'e hyper-regulus of $\S$, or in one of the two known switching sets of an Andr\'e hyper-regulus of $\S$. 
Thus there are no other planes of $\PG(5,q)$ that meet each plane of an Andr\'e hyper-regulus.
\end{proof}

\end{document}